\documentclass[12pt,oneside]{amsart}
\usepackage{amsmath}
\usepackage{amsthm}
\usepackage{amsfonts}
\usepackage{amssymb}
\usepackage{color}
\usepackage{enumerate}
\newtheorem{theorem}{Theorem}
\newtheorem{lemma}[theorem]{Lemma}
\newtheorem{proposition}[theorem]{Proposition}
\newtheorem{corollary}[theorem]{Corollary}

\theoremstyle{definition}
\newtheorem{definition}[theorem]{Definition}
\newtheorem{example}[theorem]{Example}

\theoremstyle{remark}
\newtheorem{remark}[theorem]{Remark}

\def\textmatrix#1&#2\\#3&#4\\{\bigl({#1 \atop #3}\ {#2 \atop #4}\bigr)}

\def\Aut{\operatorname{Aut}}

\def\reg{\operatorname{reg}}
\def\Int{\operatorname{int}}

\def\C{\mathbb{C}}

\def\D{\mathbb{D}}

\def\T{\mathbb{T}}

\newcommand{\DD}{{\mathbb D}}

\renewcommand{\phi}{\varphi}


\begin{document}

\title{
Extension property and universal sets}

\address{Institute of Mathematics, Faculty of Mathematics and Computer Science, Jagiellonian
University,  \L ojasiewicza 6, 30-348 Krak\'ow, Poland}

\author{\L ukasz Kosi\'nski}\email{lukasz.kosinski@uj.edu.pl}

\author{W\l odzimierz Zwonek}\email{wlodzimierz.zwonek@uj.edu.pl}
\thanks{The first author is partially supported by NCN grant SONATA BIS no.  2017/26/E/ST1/00723.
The second author is partially supported by the OPUS grant no. 2015/17/B/ST1/00996 of the National Science Centre, Poland}

\begin{abstract} 
Motivated by works on extension sets in standard domains we introduce a notion of the Carath\'eodory set that seems better suited for the methods used in 
proofs of results on characterization of extension sets. A special stress is put on a class of two dimensional submanifolds in the tridisc which not only turns out to be Carath\'eodory but also provides examples of two dimensional domains for which the celebrated Lempert Theorem holds. Additionally, a recently introduced notion of universal sets for the Carath\'eodory extremal problem is studied and new results on domains admitting (no) finite universal sets are given. 
\end{abstract}

%\section{Definitions and results}

%\subsection{Interpolation problems. Extremal maps}

\maketitle

\section{Extension property and Carath\'eodory sets}
\subsection{Introduction and state of affairs}
For a set $V\subset D$, where $D$ is a domain in $\C^n$ we denote by $\mathcal O(V)$ the set of {\it holomorphic functions} defined on $V$ as the set of all $f:V\to\C$ such that for arbitrary $w\in V$ there are an open neighborhood $W$ of $w$ in $D$ and a holomorphic $g:W\to\C$ such that $f$ coincides with $g$ on $V\cap W$. By $H^{\infty}(V)$ we mean the algebra of bounded holomorphic functions on $V$. In what follows many results could be formulated and proven for any algebras of holomorphic functions on $V$ containing polynomials; however we restrict ourselves to the special case of the algebra of bounded holomorphic functions. Additionally, for the simplicity of formulations and clarity of presentation we always assume that the set $V$ has analytic structure in the sense that $V$ is always to be assumed to be an analytic set in the given domain $D$. This means that $V$ is relatively closed in $D$ and for every point $x\in V\subset D$ there exist an open set $U\subset D$ containing $x$, and $f_1,\ldots,f_m\in \mathcal O(U)$ such that $V\cap U = \{ z \in U:\ f_j(z) = 0,\ j=1,\ldots, m\}$.

The analytic set $V\subset D$ {\it has the extension property} if for any $f\in H^{\infty}(V)$ there is an $F\in H^{\infty}(D)$ such that $F\equiv f$ on $V$ and $||F||_D=||f||_V$.

The origin of the problem of the existence of norm preserving extensions of bounded holomorphic functions goes back to Rudin's book (\cite{Rud 1969}, Theorem 7.5.5). The key step in that area of research can be found in \cite{Agl-McC 2003} where the problem was solved for the bidisc. More precisely, the following result was proven.

\begin{theorem}{\rm (see \cite{Agl-McC 2003})}\label{theorem:bidisc-annals} Let $V$ be a relatively polynomially convex subset of $\D^2$. Then $V$ has the extension property if and only if it is a retract.
\end{theorem}
Recall that $V\subset D$ is a {\it retract} if there is a holomorphic map $r:D\to D$ such that its range is $V$ and $r_{|V}$ is the identity. It is an obvious observation that any retract has the extension property.

Later Kosi\'nski and McCarthy proved, relying on the Lempert theory, that the same statement as above holds for the class of two-dimensional strictly convex domains $D$ (see \cite{Kos-McC 2018}). They also showed some necessary form of the sets with the extension property in sufficiently smooth strongly linearly convex domains in higher dimensions. Such sets must be totally geodesic.
 
On the other hand in the paper \cite{Agl-Lyk-You 2018} the authors described the sets with the extension property in the symmetrized bidisc and they found out that in this case there are sets with the extension property that are not retracts.

Although the problem of the characterization of the extension sets in the simplest case of the polydisc has been studied it is very frustrating that only some partial results on that topic were obtained. In this context let us mention results in \cite{Guo-Hua-Wan 2008}, \cite{Bha}, \cite{Mac 2019}, \cite{Kos-McC 2019}.

The situation in the tridisc $\D^3$ was studied in \cite{Kos-McC 2019}.

As one looks at the proofs of results describing the extension sets in a series of papers in quite different situations a weaker form of the extension property is more natural to work with. Namely, the existence of norm preserving extensions of some extremal functions is essential. And in principle the results on the description of extension sets just mentioned may be generalized to that new notion. This is formally done in Section~\ref{section:improvement} where the proofs are given in a detailed and partially novel way in the bidisc only; in other cases they are merely outlined.

Studying the tridisc $\D^3$ the authors showed that one dimensional sets with the extension property are precisely retracts (that was later generalized to arbitrary polydiscs in \cite{Mac 2019}) and for two-dimensional sets with the extension property they found a necessary condition (\cite{Kos-McC 2019}, Theorem 6.1). 
The last form was, as indicated in \cite{Kos-McC 2019}, not sufficient for the set to have the extension property. Then authors considered a class of two-dimensional subsets that are uniqueness varieties for three dimensional and non-degenerate 3-point Pick interpolation problem in $\D^3$ (see \cite{Kos-McC 2019}, Remark 7.4 and \cite{Kos 2015}).  In our paper we shall show that these two-dimensional subsets do satisfy our new notion, but they are not retracts. This is surprising as such a phenomenon occurs neither in the bidisc nor in domains studied in the literature so far. Therefore, this class of two-dimensional algebraic subsets is one of the objects that attracts our attention. More precisely, we look at 
analytic submanifolds $M_{\alpha}$ defined as the sets
$$
\{(z_1, z_2, z_3)\in \mathbb D^3:\ \alpha_1 z_1 + \alpha_2 z_2 + \alpha_3 z_3 = \bar \alpha_1 z_2 z_3 + \bar \alpha_2 z_1 z_3 + \bar \alpha_3 z_1 z_2\},
$$ 
where $\alpha_1, \alpha_2 , \alpha_3\in \mathbb C$ are not all zeros.

\begin{remark}\label{rem:euiv}

i) Recall that a characterization of retracts in the polydisc as the sets being graphs of holomorphic functions over lower dimensional polydiscs comes from \cite{Hea-Suf 1981}. Note that if $\alpha_3\neq 0$, the surface $M_\alpha$ can be written as a graph of a function given by the formula \begin{equation}\label{eq:z_3} z_3 = \psi(z_1, z_2)=\omega \frac{ a z_1 + b z_2 - z_1 z_2}{\bar b z_1 + \bar a z_2 -1},\end{equation}
where $a= \alpha_1/\bar\alpha_3$, $b=\alpha_2/\bar\alpha_3$ and $\omega=\bar \alpha_3/\alpha_3 $. In particular, $M_\alpha$ is a graph of a function over the first two variables  from $\D^2$ if and only if $|a|+|b|\leq 1$ that is $|\alpha_1|+|\alpha_2|\leq |\alpha_3|$.
Therefore, the above-mentioned result of Heath and Suffridge \cite{Hea-Suf 1981} implies that the variety $M_\alpha$ is not a retract of the tridisc exactly when $|\alpha_{i_1}|+ |\alpha_{i_2}| > |\alpha_{i_3}|$ for all possible permutations $(i_1, i_2, i_3)$ of the set $\{1,2,3\}$ -- we shall say that such a {\it triple satisfies the triangle inequality}. 

ii) The family $\{M_{\alpha}\}$ is stable under automorphisms of $\D^3$ in the following sense: if $z\in M_\alpha$ and $m\in \Aut(\D^3)$ maps $z$ to $0$, then $m(M_\alpha) = M_\beta$ for some $\beta$. To prove this it is enough to consider the case $\alpha_3\neq 0$ (permute the coordinates, if necessary). Let us represent $M_\alpha$ as in \eqref{eq:z_3}. Note that the function $\psi$ is \emph{inner} in the following sense: $|\psi(z_1,z_2)|=1$ for almost all $z_1,z_2$ in the unit circle. Trivial computations show that $m$ transforms $M_\alpha$ to a surface of the form 
\begin{equation}\label{eq:z_3phi}
 z_3= \varphi(z_1,z_2):=\frac{A z_1 + B z_2 + C z_1 z_2}{D z_1 + E z_2 + F},
\end{equation}
where $F\neq 0$. It can be assumed that $F=1$. It is also simple to observe (properties of $m$) that $\varphi$ is inner. This fact is crucial for the rest of our reasoning.

The following observation is trivial: if 
\begin{equation}\label{eq:inner}  \lambda\mapsto \frac{\gamma_1 \lambda + \gamma_2}{1+ \gamma_3 \lambda}\end{equation} is inner (that is maps almost all points from the unit cirle to the unit circle), then $\gamma_3 = \gamma_1 \bar \gamma_2$ and either $|\gamma_1|=1$ or $|\gamma_2|=1$.

Let us apply this observation to $\lambda \mapsto \varphi(\lambda, \omega \lambda)$, where $\omega$ is a unimodular constant (almost arbitrary). If $\gamma_1$ appearing in \eqref{eq:inner} is unimodular, we get that $|C|=1$ and 
\begin{equation}\label{eq:z_3b} \varphi(z_1, z_2) = C \frac{\bar E z_1 + \bar D z_2 + z_1 z_2}{D z_1 + E z_2+1}.\end{equation}
If, in turn, $|\gamma_2|=1$, then either $B$ is unimodular and $\varphi(z_1, z_2) = B z_2$ or $A$ is unimodular and $\varphi(z_1, z_2) = A z_1$. Certainly the last two cases cannot occur (otherwise $M_\alpha$ would be of the form $z_3 = \omega_j z_j$ for some $j=1,2$). In particular, $m(M_\alpha)$ is of the form \eqref{eq:z_3b} and consequently can be written as $M_\beta$ for $\beta = (c \bar E, c \bar D, - \bar c)$, where $c^2=C$.

Note also that $\beta$ satisfies the triangle inequality if and only if $\alpha$ does. This is an immediate consequence of the fact that $m$ takes holomorphic retracts to holomorphic retracts and the previously mentioned description of such sets.

iii) All the results for $M_{\alpha}$ presented below are non-trivial precisely when the triple $\alpha$ satisfies the triangle inequality -- otherwise the sets are biholomorphic to the bidisc.

\end{remark}

The notion that is new and is basic in our paper is that of an (infinitesimally) Carath\'eodory set to be defined below. 
To define the objects we use notion of the Carath\'eodory pseudodistance (Carath\'eodory-Reiffen pseudometric) that is an example of a holomorphically invariant function. Since basic properties of holomorphically invariant functions are essential for us we ask the Reader to consult the book \cite{Jar-Pfl 2013} on the fundamental properties of these functions.

\bigskip

By $\rho$ we denote the hyperbolic metric on the unit disc $\D:=\{\lambda\in\C:|\lambda|<1\}$. We also denote $\T:=\partial\D$.

We define {\it the Carath\'eodory pseudodistance} 
$$c_V(z,w):=\sup \{\rho(F(z), F(w)):\ F\in \mathcal O(V, \mathbb D)\}.$$

Its infinitesimal version, {\it the Reiffen-Carath\'eodory pseudometric}, is defined below
$$
\gamma_V(w;X):=\sup\{|F^{\prime}(w)X|:F\in H^{\infty}(V),||F||\leq 1, F(w)=0\},
$$
where $w\in V$ is a regular point ($w\in V_{reg}$) and $X$ is an arbitrary vector from the tangent space $T_wV$.
 
A function $F$ for which the supremum in the definition above is attained is called {\it extremal} (respectively, {\it infinitesimally extremal}) for the pair $(z,w)$ 
(respectively, $(w;X)$).

Note that if $V$ is an analytic set in $D$ then $c_V(w,z)\geq c_D(w,z)$ for any $w,z\in V$ and $\gamma_V(w;X)\geq \gamma_D(w;X)$ for any $w\in V_{reg}$, $X\in T_wV$.

\begin{definition} Let $V$ be a set (an analytic variety) in a subdomain $D$ of $\C^n$. 

We say that $V$ is a {\it Carath\'eodory set} if $$c_D(z,w) = c_V(z,w)\quad \text{ for all } z,w\in V.$$

We say that $V$ is an {\it infinitesimal Carath\'eodory set} if $$\gamma_D(w;X)=\gamma_V(w;X)$$ for any regular point $w\in V_{reg}$ and $X\in T_wV$.

\end{definition}
 
As we have already announced the (infinitesimal) Carath\'eodory sets are the ones that admit the norm preserving extensions of (infinitesimally) extremal functions.

It is an elementary observation that if $V$ has the extension property then it is a Carath\'eodory set. Any Carath\'eodory set is an infinitesimal Carath\'eodory set. In Section~\ref{section:improvement} we briefly sketch how the known proofs on results describing extension sets presented above apply to the situation of the Carath\'eodory sets.  At this place note that any Carath\'eodory set must be connected.

\subsection{Link to the Lempert theory}
While trying to characterize the sets having the extension property (or Carath\'eodory sets) the impact of the Lempert theory of holomorphically invariant functions in convex domains turns out to be essential. First recall that for the domain $D$ and points $w,z\in D$ we define {\it the Lempert function}

\begin{multline}
l_D(w,z):=\\
\inf\{\rho(0,t):\exists f\in\mathcal O(\D,D) \text{ such that } f(0)=w,f(t)=z\}.
\end{multline}

Note that the definition of the Lempert function may be easily extended to arbitrary subsets $M\subset\C^n$ (with $D$ replaced by $M$) - in case there is no holomorphic mapping joining $w$, $z$ lying entirely in $M$ we define $l_M(w,z):=\infty$. We may also define the infinitesimal version of the Lempert function. We define
\begin{equation}
\kappa_V(w;X):=\inf\{|\lambda|:\exists f:\D\to V,f(0)=w,\lambda f^{\prime}(0)=X\},
\end{equation}
where $w\in V_{reg}$ and $X\in T_w V$. The function $\kappa_V$ is called {\it the Kobayashi-Royden pseudometric}.

We call a holomorphic mapping $f:\D\to D$ a {\it complex geodesic} if there is a holomorphic function $F:D\to\D$ such that $F\circ f$ is an automorphism of the unit disc. In particular, for any $\lambda,\mu\in\D$ we have the equality $\rho(\lambda,\mu)=c_D(f(\lambda),f(\mu))$, $\lambda,\mu\in\D$. We call the function $F$ to be {\it the left inverse of } the complex geodesic $f$. We also say that {\it the complex geodesic $f$ passes through} $w,z\in D$ (respectively, $(w;X)\in D\times\C^n$) if $w,z\in f(\D)$ (respectively, $w=f(\lambda)$ and $X$ is parallel to $f^{\prime}(\lambda)$ for some $\lambda\in\D$).

The main result in the Lempert theory is the following.

\begin{theorem}{\rm (see \cite{Lem 1981}, \cite{Lem 1982}).} Let $D$ be a convex domain in $\C^n$. Then $c_D\equiv l_D$. Moreover, if $D$ is also bounded, then for any two points $w,z\in D$ (respectively, $w\in D$, $X\in\C^n$) we may find a complex geodesic $f:\D\to D$ passing through $w,z$ (respectively, $(w;X)$).
\end{theorem}

Actually for taut domains $D\subset\mathbb C^n$ (i.e. such that any sequence of holomorphic mappings $f_k:\D\to D$ is a normal family) the fact that $c_D$ coincides with $l_D$ is equivalent to the existence for any pair of points $w,z\in D$ of a complex geodesic passing through them. It is quite natural to call any taut domain $D$ such that $c_D\equiv l_D$ {\it a Lempert domain}. The Lempert theorem states that any bounded convex domain is a Lempert domain. Note that the complex geodesics in $D$ are proper holomorphic embeddings of the unit disc into the domain $D$.

\begin{definition} Let $V$ be an analytic subset of a domain $D$ in $\C^n$.

We say that $V$ is {\it totally geodesic} in $D$ if for any $z,w\in V$, $z\neq w$ there exists a complex geodesic $f:\mathbb D\to D$ passing through $z$ and $w$ that lies entirely in $V$.

We say that an analytic set $V$ in $D$ is {\it infinitesimally totally geodesic} if for any $w\in V_{reg}$ and any $X\in T_wV$ we find a complex geodesic $f:\D\to D$ such that $f(0)=w$, the vector $X$ is parallel to $f^{\prime}(0)$ and the image of $f$ lies entirely in $V$. 
\end{definition}

\begin{proposition}\label{prop:tot} Any (infinitesimally) totally geodesic set $V$ in a Lempert domain $D$ is an (infinitesimally) Carath\'eodory set.
\end{proposition}
\begin{proof} Choose distinct points $z,w\in D$ and a holomorphic function $F:V\to \DD$. Let $f$ be a complex geodesic passing through $z,w$ that lies in $V$. 
Choose $\lambda,\mu\in \DD$ so that $f(\lambda) = z$, $f(\mu)=w$. Then, by the Schwarz lemma, 
$$
\rho(F(z), F(w)) = \rho(F(f(\lambda), F(f(\mu)) \leq \rho(\lambda,\mu) = l_D (z,w) =c_D(z,w).
$$
Consequently, $c_V(z,w)\leq c_D(z,w)$ which finishes the proof in the first case. The proof in the infinitesimal case goes along the same lines so we skip it.
\end{proof}

\subsection{Main results} 
As we already announced the following theorem is one of the main results of our paper (see Theorem~\ref{thm:Viscar}):
\emph{the set $M_\alpha=\{z\in\mathbb D^3:\alpha_1 z_1+\alpha_2 z_2+ \alpha_3 z_3=\bar \alpha_3 z_1z_2+\bar \alpha_2 z_1z_3+\bar \alpha_1 z_2z_3\}$ is Carath\'eodory. Moreover, $$l_{M_{\alpha}}\equiv c_{M_{\alpha}}\equiv (c_{\D^3})_{|M_{\alpha}^2}.$$}

Although we do not know whether the sets $M_{\alpha}$ are always extension sets the above result gives not only a new insight into the understanding of the extension property and extends results of \cite{Kos-McC 2019} but also provides an interesting class of Lempert domains. To make the last statement clear we introduce a class of two dimensional subdomains in $\D^2$. For $a,b>0$ we define
\begin{equation}
D_{a,b}:=\{z\in\D^2:|F_{a,b}(z_1,z_2)|<1\},
\end{equation}
where $$F_{a,b}(z_1,z_2):=\frac{az_1+bz_2-z_1z_2}{az_2+bz_1-1}.$$ Note that the function $F_{a,b}$ is the one which gives a solution $z_3(z_1,z_2)$ of the equation defining $M_{\alpha}$ for suitably chosen $\alpha$. Similarly as above the domain $D_{a,b}$ is interesting precisely when the triple $\{a,b,1\}$ satisfies the triangle inequality.

A direct consequence of the above result is the following (see Theorem~\ref{theorem:lempert-theorem-example}): \emph{the domain $D_{a,b}$ is a Lempert domain.}

As we shall see later the domains $D_{a,b}$ are, under some obvious assumptions on $a$ and $b$, even not linearly convex (see Remark~\ref{remark:not-linearly-convex}). The existence of such a class of domains is interesting from the point of view of the Lempert theory, as the only domain  with all holomorphically invariant functions equal for which that equality could not be concluded from the Lempert theorem is the tetrablock (introduced in \cite{AWY}) --  see \cite{Edi-Kos-Zwo 2013}. Recall that the symmetrized bidisc was, at the time of its discovery, the first example posessing such a phenomenon  (see \cite{Cos 2004}, \cite{Agl-You 2004}). It turned out later that this domain can be exhausted by strongly linearly convex domains - that made possible to deduce the equality of all holomorphically invariant functions on the symmetrized bidisc from the Lempert theorem (see \cite{Pfl-Zwo 2012}).
%The result is no longer true in higher dimensional polydisc:
%\begin{example}
%Let $M=M_{\alpha}$ denote the set $$\{(z_1, z_2, z_3)\in \mathbb D^3:\ \alpha_1 z_1 + \alpha_2 z_2 + \alpha_3 z_3 = \bar \alpha_1 z_2 z_3 + \bar %\alpha_2 z_1 z_3 + \bar \alpha_3 z_1 z_2\},$$ where $\alpha_1, \alpha_2 , \alpha_3\in \mathbb C$. Then $M$ is a Carath\'eodory set.
%\end{example}

%\begin{remark}\label{rem:euiv}
%Note that the variety $M_\alpha$ is not a retract of the tridisc exactly when $|\alpha_{i_1}|+ |\alpha_{i_2}| > |\alpha_{i_3}|$ for all $(i_1, i_2, i_3)$. %Moreover, the family $\{M_{\alpha}\}$ is stable under automorphisms of $\D^3$ in the following sense: if $z\in M_\alpha$ and $\varphi\in \Aut(\D^3)$ %maps $z$ to $0$, then $\varphi(M_\alpha) = M_\beta$ for some $\beta$ (this can also be deduced from the fact that $M_\alpha$ are uniqueness varieties %for a $3$-point Pick interpolation problem in the polydisc).
%\end{remark}

\bigskip

In Section~\ref{section:universal-set} motivated by a recent paper \cite{Agl-Lyk-You 2018} we consider the universal sets for the Carath\'eodory problem, i.e. the sets $\mathcal C\subset\mathcal O(D,\D)$ which can replace the set $\mathcal O(D,\D)$ in the definition of the Carath\'eodory pseudodistance of the domain $D$. We remark that in dimension one under very mild and natural assumptions the existence of a finite universal set for the Carath\'eodory problem implies that the domain is the disc (Theorem~\ref{thm:extr-planar}). 

It is noted in \cite{Agl-Lyk-You 2018} that in dimension two the existence of a universal set with two elements requires the domain to be the bidisc. We generalize this result showing that the existence of a finite universal set lets the domain embed in the polydisc (possible of higher dimension) -- see Theorem~\ref{the:univer-fin}. Additionally, the domains $D_{a,b}\subset\C^2$ turn out to be examples of the ones admitting three and not two elements in the universal set. Therefore, the situation in dimension two differs from that in dimension one and there are other than the bidisc nice and non-trivial domains admitting a finite number of elements in the universal set - see Example~\ref{example:finite-universal-set}.

In Section~\ref{section:improvement} the main stress is put on extension and simplification of the situation in the bidisc, i.e. Theorem~\ref{theorem:bidisc-annals}, which is the content of Theorem~\ref{th:infbid} and Corollary~\ref{th:bidisc}. The characterization of Carath\'eodory sets in the situation of strictly convex, strongly linearly convex and the symmetrized bidisc is sketched only as the methods from \cite{Agl-Lyk-You 2019}, \cite{Kos-McC 2018} and \cite{Kos-McC 2019} apply in the general case word by word.

In Section~\ref{section:ball} we briefly discuss the problem of a possible structure of universal sets for the Carath\'eodory extremal problem in the Euclidean ball $\mathbb B_n$ and we show how to produce universal sets that are 'smaller' than the ones obtained by the most evident way.

\section{Carath\'eodory sets replace extension sets}\label{section:improvement}
%For the complex variety $V\subset D$ we define {\it the infinitesimal Carath\'eodory set} to be the one satisfying the property
%\begin{equation}
%\gamma_V(z;X)=\gamma_D(z;X)
%\end{equation}
%for any $z$ in the regular part of $V$ and $X$ from the complex tangent space $T_zV$.
This section studies connection between Carath\'eodory and extension sets. Its first aim is to describe Carath\'eodory sets in the bidisc showing that they are holomorphic retracts. It is thus trivial that both notions coincide there which, in particular, proves and extends \cite{Agl-McC 2003}. After a more detailed study of the case of the bidisc we sketch how the proofs in \cite{Agl-Lyk-You 2019}, \cite{Kos-McC 2018} and \cite{Kos-McC 2019} may be applied to get the results describing Carath\'eodory sets in the cases of strictly convex, strongly linearly convex domains and the symmetrized bidisc --- the proofs of (formally stronger) results follow exactly the same lines as in appropriate papers.

\begin{proposition}\label{prop:left} Let $D$ be a domain in $\mathbb C^n$ and $V\subset D$ an analytic set. Assume that $c_V(z,w) = c_D(z,w)> 0$ for some $z,w\in V$ (respectively, $\gamma_D(z;X) = \gamma_V(z;X)>0$, for some $z\in V_{reg}$ and $X\in T_zV$). If the function $F:D\to \mathbb D$ is extremal for $(z,w)$ (resp. for $(z;X)$), then $F(V)$ is dense in $\mathbb D$.
\end{proposition}

\begin{proof} Suppose that $F(V)$ is not dense in $\D$. Let $\emptyset\neq\bar\triangle\subset\D$, $\bar\triangle \cap F(V)=\emptyset$, be a closed disc.
In the first case the result follows from Lemma 9.3 in \cite{Agl-Lyk-You 2019} which gives the existence of a holomorphic function 
$\beta:\DD\setminus \bar\Delta\to \DD$ such that 
$$
\rho(F(z), F(w))< \rho(\beta(F(z)), \beta(F(w)))\leq c_D(z,w),
$$ which contradicts the extremality of $F$. The infinitesimal case follows the same idea (and makes use of the same Lemma 9.3 from \cite{Agl-Lyk-You 2019}). 
\end{proof}

Below we use the notion of balanced points introduced in \cite{Agl-McC 2003}.
A pair $(z;w)\in \D^2\times\D^2$, $z\neq w$ (respectively, $(z;X)\in\D^2\times\C^2$, $X\neq (0,0)$) is called {\it balanced} (respectively, {\it infinitesimally balanced}) if it satisfies the equality $\rho(z_1,w_1)=\rho(z_2,w_2)$ (respectively,
$\gamma_{\D}(z_1;X_1)=\gamma_{\D}(z_2;X_2)$). Note that being (infinitesimally) balanced is invariant under holomorphic automorphisms of $\D^2$ in the sense that if the pair $(z,w)$ (respectively, $(z;X)$) is (infinitesimally) balanced  then so is $(a(z);a^{\prime}(z)(X))$ (respectively, $(a(z),a(w))$) for any automorphism $a$ of $\D^2$. It follows from the Schwarz lemma that the (infinitesimally) balanced pair determines uniquely a complex geodesic passing through them -- in this case both components of the geodesic are automorphisms of the unit disc). 

\begin{lemma}\label{lemma:balanced} Let $V\subset\D^2$ be polynomially convex analytic subvariety. If there is a balanced pair $(z,w)$ such that $z,w\in V$ and $c_V(z,w) = c_{\DD^2}(z,w)$ (respectively, infinitesimally balanced pair $(z,X)\in V_{reg}\times T_zV$ and $\gamma_V(z;X) = \gamma_{\DD^2}(z;X)$) then $V$ contains the unique complex geodesic passing through them.
\end{lemma}
\begin{proof} The proof is standard, we are recalling it for the sake of completeness.
Losing no generality we can assume that $z=(0,0)$.  Since $((0,0),w)$ (respectively, $((0,0);X)$) forms a balanced (respectively, infinitesimally balanced) pair we get  that $w_2=\omega w_1$ (respectively, $X_2=\omega X_1$) , $|\omega|=1$. 
Then making use of the fact that the function $\frac{z_1+\bar\omega z_2}{2}$ is (infinitesimally) extremal for $((0,0),w)$ (respectively, $((0,0);X)$), due to Proposition~\ref{prop:left} we get that the set $\{\frac{z_1+\bar \omega z_2}{2}:z\in V\}$ is dense in $\D$, so the polynomial convexity of $V$ implies that $\{(\lambda,\omega\lambda):\lambda\in\D\}\subset V$, which finishes the proof.
\end{proof}

\begin{theorem}\label{th:infbid} Let $V$ be an analytic subvariety of $\mathbb D^2$ that is polynomially convex. Then $V$ is an infinitesimal Carath\'eodory set if and only if it is a union of a discrete set and complex geodesics.
\end{theorem}

\begin{proof}
The only difficult part is to show that any one-dimensional irreducible component $W$ of an infinitesimal Carath\'eodory set is (equivalently, contains) a complex geodesic. 

The case when there is a point $(z;X)\in W_{reg}\times (T_z W\setminus\{0\})$ that is infinitesimally balanced follows from Lemma~\ref{lemma:balanced}.

Suppose that no pair $(z;X)\in W_{reg}\times(T_z W\setminus\{0\})$ is balanced. Then without loss of generality we may assume that
\begin{equation}\label{eq:unbal} \gamma_{\D}(z_2,X_2) < \gamma_{\D}(z_1, X_1) \text{ for every }z\in W_{\reg} \text{ and }X\in T_z(W)\setminus \{0\}.
\end{equation} 
In particular, near every such $z$ the variety $W$ is a graph of a holomorphic function over the first coordinate. 
Moreover, 
\begin{equation}\label{eq:unbal-conclusion}
\rho(z_2, w_2)< \rho(z_1, w_1)
\end{equation}
 if $z,w\in W_{reg},$ $z\neq w$ are close enough to each other.

We shall prove that $W$ contains a graph of a holomorphic function $f:\D\to\D$ with the property $|f(\lambda)|<|\lambda|$, $\lambda\in\D\setminus\{0\}$ which would finish the proof.

First note that near $(0,0)$ we get the existence of a holomorphic $f_0:\triangle(0,\epsilon)\to\D$ with the property $|f_0(\lambda)|<|\lambda|$, $\lambda\in\triangle(0,\epsilon)\setminus\{0\}$ and $\{(\lambda,f_0(\lambda)):\lambda\in\triangle(0,\epsilon)\}=W\cap \triangle(0,\epsilon)^2$. Below we show how we may extend the function to the whole unit disc. Denote $K:=\pi_1((W\setminus W_{reg})\cap\{|z_2|\leq|z_1|\})$ which is a discrete subset of $\D\setminus\triangle(0,\epsilon)$. 

Define $\mathcal F$ as the family of all pairs $(U,g)$ where $\triangle(0,\epsilon)\subset U\subset\D$, $U$ is star-shaped (with respect to $0$), $g:U\to\D$ is holomorphic with $|g(\lambda)|<|\lambda|$, $\lambda\in U\setminus\{0\}$ and $g|_{\triangle}\equiv f_0$. The identity principle shows that for any two pairs $(U_j,g_j)\in\mathcal F$ we have
$g_1=g_2$ on $U_1\cap U_2$. Consequently, the relation $(U_1,g_1)\leq (U_2,g_2)$ if $U_1\subset U_2$ is a partial order on $\mathcal F$.

We define the extension as follows $\mathcal U:=\bigcup_{(U,g)\in\mathcal F}U$ and the function $f:\mathcal U\to \D$ with the formula $f(\lambda):=g(\lambda)$ if $\lambda\in U$ and $(U,g)\in\mathcal F$. Note that the function $f$ is well-defined, holomorphic, coincides with $f_0$ on $\triangle(0,\epsilon)$ and $|f(\lambda)|<|\lambda|$, $\lambda\in \mathcal U\setminus\{0\}$. The element $(\mathcal U,f)$ is a maximal element of $\mathcal F$. It is sufficient to show that $\mathcal U=\D$. Suppose the opposite. Take a point $\mu\in\partial \mathcal U\cap\D$ such that the ray $\{r\mu:r\in[0,1)\}\subset \mathcal U$.  First note that $f(\mu):=\lim_{\mathcal U\owns\lambda\to\mu}f(\lambda)$ exists and $|f(\mu)|\leq|\mu|$, as otherwise $W$ would contain over $\mu$ uncountably many points so the disc $\{\mu\}\times\D$ would be contained in $W$ which contradicts the irreducibility of $W$.  Note also that $|f(\mu)|<|\mu|$. To see it divide $[0,1]$ into sufficiently small intervals $0=t_0<t_1<\cdots <t_N:=1.$ Using \eqref{eq:unbal-conclusion} we get $$\rho(f(\mu), f(0))\leq \sum_{j=1}^N \rho(f(t_{j-1}\mu), f(t_{j}\mu)) < \sum_{j=1}^N\rho(t_{j-1}\mu, t_{j}\mu) = \rho(\mu,0).$$ 

We claim that $\mu\in K$. Suppose the opposite. Then $(\mu,f(\mu))\in W_{reg}$. The set $W_{reg}$ is near the point $(\mu,f(\mu))$ the graph of a holomorphic function which easily gives a strictly bigger element than $(\mathcal U,f)$, which contradicts its maximality.

Now we extend the function $f$ as follows. We already know that $\Int(\bar {\mathcal U})$ is $\D$. And now proceeding similarly as earlier assuming that $f$ cannot be extended continuously to $\D$ we get that at some point $\mu$ from $\partial \mathcal U\cap\D$ the cluster set of $f$ would contain uncountably many points from $\D$ which would force $W$ to contain $\{\mu\}\times\D$ -- a contradiction. The continuous extension of $f$ is then trivially holomorphic.

\end{proof}

\begin{corollary}\label{th:bidisc}
Let $V$ be as in Theorem~\ref{th:infbid}. Then $V$ is a Carath\'eodory set if and only if it is a holomorphic retract.
\end{corollary}

\begin{proof} We already know that the Carath\'eodory set is connected.
If $V$ is not a single point, then, by Theorem~\ref{th:infbid}, $V$ contains the graph of a complex geodesic $\mathcal G$ which is, up to a permutation of variables, of the form $\{(\lambda, f(\lambda)):\lambda\in \mathbb D\}$, where $f\in \mathcal O(\DD, \DD)$. If $V$ is not equal to $\mathcal G$, we can assume that $(0,0)\in V$ and $f(0)\neq 0$. Note that $|f(\lambda)|<|\lambda|$ for $\lambda\in \DD\setminus \{0\}$ implies that $f(0)=0$ which contradicts our assumption, while $|f(\lambda)|>|\lambda|$ for $\lambda\in \mathbb D_*$ is impossible. Therefore, there is at least one $\mu$ (and thus uncountable many) such that $|f(\mu)|=|\mu|$. For every such $\mu$ the pair $(0,0)$ and $(\mu, f(\mu))$ is balanced. Therefore these $\mu$ the variety $V$ contains a geodesic $\{(\lambda, \omega_\mu \lambda):\lambda\in \DD\}$, where $\omega_\mu$ satisfies $f(\mu) = \omega_\mu \mu$. It is also elementary that the set of all these $\omega_{\mu}$ is uncountable. From this we simply get that $V$ is the whole bidisc.
\end{proof}

The equivalence of the notions of the Carath\'eodory set and infinitesimal Carath\'eodory set should hold for varieties without singular points for a reasonable class of $V\subset D$.

\bigskip

In the remaining part of the section we shall describe relation between Carath\'edory and extension sets in some classes of domains. Roughly speaking both notions can be replaced by each other in statements of all the results that have been obtained so far. We shall explain this briefly below.

Recall that a domain $D\subset\C^n$ is {\it linearly convex} if its complement is the union of complex affine hyperplanes. A domain $D$ with the smooth defining function $r$ satisfying the inequality
\begin{equation}
\sum_{j,k=1}^n\frac{\partial^2 r}{\partial z_j\partial\bar z_k}(w)X_j\bar X_k>\left|\sum_{j,k=1}^n\frac{\partial^2 r}{\partial z_j\partial z_k}(w)X_jX_k\right|
\end{equation}
for any $w\in\partial D$, $X\neq 0$ lying in the complex tangent hyperplane to $\partial D$ at $w$ is called {\it strongly linearly convex}.

\begin{remark} Using methods from the present paper and the ones used in \cite{Kos-McC 2018} we can get the following result:
Let a domain $D$ in $\mathbb C^n$ be strictly convex or strongly linearly convex. Let $V$ be a relatively polynomially convex analytic subset of $D$. Then $V$ is a Carath\'eodory set if and only if it is totally geodesic.

In particular, if $D$ is the Euclidean ball $\mathbb B_n$ or if $n=2$, then any Carath\'eodory set $V$ is a holomorphic retract.
\end{remark}

\begin{remark}
Except for examples described above the extension property problem was solved fully only in a particular example of the domain 
 \begin{equation}
\mathbb G_2:=\{(\lambda+\mu, \lambda \mu):\ \lambda,\mu\in \DD\}
\end{equation}
called {\it the symmetrized bidisc}. It turns out that both notions coincide there in a reasonable class of domains. More precisely,  an algebraic set $V$ in  $\mathbb G_2$ is a Carath\'eodory set if and only if it has extension property. Moreover, there are one dimensional Carath\'eodory sets in $\mathbb G_2$ that are not complex geodesics.

 To see this one needs to follow the proof in \cite{Agl-Lyk-You 2019} to get that any Carath\'eodory set $V$ is either $V_\beta$ or $V_\beta \cup \Sigma$, $\beta\in \DD$, $V_\beta=\{(\beta + \bar \beta \lambda,\lambda): \lambda \in \mathbb D\}$ and $\Sigma=\{(2\lambda, \lambda^2):\lambda\in \DD\}$ or there is a biholomorphic mapping $\iota:\DD\to V$. In the latter case, $\rho(\iota^{-1}(z), \iota^{-1}(w))\leq c_{V}(z,w) = c_{\mathbb G_2}(z,w)$ for $z,w\in V.$ Thus for $\iota(\lambda)=z$ and $\iota(\mu)=w$ we get $\rho(\lambda, \mu) \leq c_{\mathbb G_2}(\iota(\lambda), \iota(\mu))\leq \rho(\lambda, \mu)$, which means that $\iota$ is a geodesic.
\end{remark}

\begin{remark}
The extension problem has been studied in the tridisc in \cite{Kos-McC 2019}, where partial characterization of extension sets were obtained. One can modify arguments used there along methods exploited within the proof of Theorems~\ref{th:infbid} and \ref{th:bidisc} to get that all the assertions of the main results for extension sets in \cite{Kos-McC 2019} are also satisfied by Carath\'eodory sets.
\end{remark}

\section{Carath\'eodory sets in the tridisc - the case of the sets $M_{\alpha}$}\label{section:sets-M}

Our first and main result of this section is the proof of the fact that the sets $M_{\alpha}$ -- two dimensional algebraic submanifolds of $\D^3$ defined earlier -- are Carath\'eodory sets. Moreover, the equality as in the Lempert theorem holds for them. The main result is formulated below.

\begin{theorem}\label{thm:Viscar}
The set $M_\alpha=\{z\in\mathbb D^3:\alpha_1 z_1+\alpha_2 z_2+ \alpha_3 z_3=\bar \alpha_3 z_1z_2+\bar \alpha_2 z_1z_3+\bar \alpha_1 z_2z_3\}$ is Carath\'eodory. Moreover, $l_{M_{\alpha}}\equiv c_{M_{\alpha}}\equiv (c_{\D^3})_{|M^2_{\alpha}}$.
\end{theorem}

 What remains unclear for us is whether the sets $M_{\alpha}$ have the extension property. Moreover, as we shall also see later in the section,  
the sets $M_{\alpha}$ give rise to a construction of two-dimensional domains (denote by $D_{a,b}$) that are Lempert. To the best of our knowledge the fact that the domains $D_{a,b}$ are Lempert cannot be proven by methods developed by Lempert. This makes the sets extremely interesting from that point of view - we shall address these problems at the end of the present Section.

\bigskip

The result is trivial if $M:=M_{\alpha}$ is a retract. Therefore, from now on we shall assume that this is not the case. In other words the inequalities $|\alpha_{j_1}|+|\alpha_{j_2}|>|\alpha_{j_3}|$ are satisfied for all permutations $(j_1,j_2,j_3)$ of the set $\{1,2,3\}$. In the sequel the triples $\alpha$ satisfying this property will be called the ones that {\it satisfy the triangle inequality}. Note that if the triple $\alpha$ satisfies the triangle inequality then $\alpha_j\neq 0$, $j=1,2,3$.

\bigskip

%Note that any $M_{\alpha}$ and $z\in M_{\alpha}$ may be mapped biholomorphically onto some other $M_{\beta}$

According to Proposition~\ref{prop:tot} to get the assertion it is sufficient to show that $M$ is totally geodesic. To prove it we shall first show in Lemma~\ref{lem:Viscar} that it is infinitesimally totally geodesic. Then a topological argument will finish the proof.

%Since the map $z\to m_{w_1}(z_1),m_{w_2}(z_2},m_{w_3}(z_3))$ maps $M_{\alpha}$ biholomorphically onto some $M_{\beta}$ and $w$ to $0$

Using Remark~\ref{rem:euiv} we may make the following reduction. Instead of proving the fact that $M_{\alpha}$ is Carath\'eodory for the fixed $\alpha$
it is sufficient to show that $$l_{M_{\alpha}}(0,x)=c_{M_{\alpha}}(0,x) = c_{\D^3}(0,x),\quad x\in M_{\alpha}$$ for any $M_{\alpha}$ being not a retract.

Then the idea of the proof of the above equality is the following. First we show the existence of a complex geodesic passing through the origin and arbitrary vector tangent to $M_{\alpha}$ at $0$. It is theoretically possible that not all the points from the set $M_{\alpha}$ will be achieved by such geodesics.
But the use of some topological argument will provide the existence of a geodesic passing through $0$ and arbitrary point of $M_{\alpha}$. At the same time we show that left inverses to all such geodesics may be attained by one of three functions: the projections of $M$ onto one of the three axes.

\bigskip

We fix now $M=M_{\alpha}$ being not a retract. Note that $M$ can be written as $$M=\left\{z\in \D^3:\ z_3 = \omega \frac{a z_1 + b z_2 - z_1 z_2}{\bar b z_1 + \bar a z_2 - 1}\right\},$$ for some numbers $a,b$ such that the triple $\{a,b,1\}$ satisfies the triangle inequality and unimodular $\omega$. Using linear automorphism of $\D^3$ we can additionally assume that $a,b> 0$ and $\omega =1$. Note that in such a case $T_0M=\{X\in\mathbb C^3:aX_1+bX_2+X_3=0\}$.

We start the proof with the following statement.

\begin{lemma}\label{lem:Viscar}  Define 
$$
\mathcal X:=\{\gamma=(\gamma_1,\gamma_2)\in\mathbb C^2:(\gamma,1)\in T_0M \text{ and } |\gamma_1|,|\gamma_2|<1\}.
$$ 
Then there are two continuous mappings  
\begin{equation}
(\omega,\eta)=(\omega,\eta)(\gamma):\mathcal X\to\T^2
\end{equation}
such that for any $\gamma\in\mathcal X$ the image of the mapping
\begin{equation}
(\dag)\qquad \Phi_{\gamma}:\D\owns \lambda\mapsto \left(\lambda m_{\gamma_1}(\omega \lambda), \lambda m_{\gamma_2}(\eta \lambda),\lambda \right)
\end{equation}
 lies in $M$ ($m_{\nu}(\lambda):=\frac{\nu-\lambda}{1-\bar\nu\lambda}$, $\lambda\in\D$, $\nu\in\D$). 

The two existing mappings $(\omega,\eta)$ are such that both components differ at each argument.

Consequently, for any $\gamma\in\mathcal X$ there are two non-equivalent geodesics (i.e. having different image) passing through the pair $(0;(\gamma,1))$. 

We also have the 'infinitesimal' version of the Lempert theorem at $0$, namely, the equality $\kappa_M(0;X)=\gamma_M(0;X)$ holds for any $X\in T_0M$.

%the third component is the Blaschke product of degree two. Moreover, we shall show that for a fixed $\gamma$ there are exactly two different $\omega$ %(respectively, $\eta$) with this property. Additionally, the number $\omega:=\omega(\gamma)$ (respectively, $\eta:=\eta(\gamma)$) may be chosen so that %$\omega$ (respectively, $\eta$) is continuous.

%Consequently, for any $X\in T_0 M$ (i.e. such that $aX_1+bX_2+X_3=0$) there is a complex geodesic in the tridisc for a pair $(0;X)$ that lies entirely in $M$.  
\end{lemma}
\begin{remark}
The set $\mathcal X$ is geometrically lens (linear transformation of intersection of two discs) with exactly two points $\gamma,\tilde\gamma$ from the closure that lie in $\T^2$.
\end{remark}
\begin{proof}[Proof of Lemma~\ref{lem:Viscar}]
% Fix $X\in\mathbb C^3$ tangent to $M$ at $0$, i.e. $aX_1+bX_2+X_3=0$. Without loss of generality we assume that $X_1=1$, $X_2=\gamma$, $X_3=-a-b%\gamma$ and $|\gamma|,|a+b\gamma|<1$. The idea is to show that there are $\omega,\eta\in T$ such that $$(\dag)\qquad \lambda\mapsto \left( 
%\lambda, \lambda m_{\gamma}(\omega \lambda), \lambda m_{-a - b \gamma}(\eta \lambda) \right)$$ lies in $M$. Moreover, we shall show that for a %fixed $\gamma$ there are exactly two different $\omega$ (respectively, $\eta$) with this property. Additionally, the number $\omega:=\omega(\gamma)$ %(respectively, $\eta:=\eta(\gamma)$) may be chosen so that $\omega$ (respectively, $\eta$) is continuous.

Let us recall that the triple $\{a,b,1\}$ satisfies the triangle inequality - it will be used in the sequel extensively.

Our aim is to show for any $\gamma\in\mathcal X$ the existence of exactly two pairs $(\omega,\eta)(\gamma)\in\T^2$ (moreover, varying continuously) such that the following equality
\begin{multline}
a\lambda\frac{\gamma_1-\omega\lambda}{1-\omega\bar\gamma_1\lambda}+b\lambda\frac{\gamma_2-\eta\lambda}{1-\eta\bar\gamma_2\lambda}+\lambda=\\
a\lambda^2\frac{\gamma_2-\eta\lambda}{1-\eta\bar\gamma_2\lambda}+b\lambda^2\frac{\gamma_1-\omega\lambda}{1-\omega\bar\gamma_1\lambda}+\lambda^2\frac{\gamma_1-\omega\lambda}{1-\omega\bar\gamma_1\lambda}\frac{\gamma_2-\eta\lambda}{1-\eta\bar\gamma_2\lambda}
\end{multline}
holds for all $\lambda\in\D$.

Keeping in mind that $\gamma\in\mathcal X$ we easily get that the above equality holds if and only if
\begin{equation}
a\omega+a\eta\bar\gamma_2\gamma_1+b\eta+b\omega\bar\gamma_1\gamma_2+\omega\bar\gamma_1+\eta\bar\gamma_2=-a\gamma_2-b\gamma_1-\gamma_1\gamma_2.
\end{equation}
Keeping in mind that $\gamma\in\mathcal X$ (more precisely making use of the equality $a\gamma_1+b\gamma_2+1=0$) we get that the last is equivalent to
\begin{equation}
a\omega(1-|\gamma_1|^2)+b\eta(1-|\gamma_2|^2)+a\gamma_2+b\gamma_1+\gamma_1\gamma_2=0.
\end{equation}
The elementary planar geometric properties show that to finish the proof it is sufficient to show that for all $\gamma\in\mathcal X$ the following inequalities hold
\begin{equation}
|a(1-|\gamma_1|^2)-b(1-|\gamma_2|^2)|<|a\gamma_2+b\gamma_1+\gamma_1\gamma_2|<a(1-|\gamma_1|^2)+b(1-|\gamma_2|^2).
\end{equation}

To prove the right inequality we consider the function
\begin{equation}
h(\gamma):=|a\gamma_2+b\gamma_1+\gamma_1\gamma_2|-a(1-|\gamma_1|^2)-b(1-|\gamma_2|^2)
\end{equation}
that is defined on a complex line $l$ containing $\mathcal X$ (that is identified as a subdomain of $l$). The function $h$ is subharmonic as a function of a complex variable. To prove the desired inequality it is sufficient to show that $h$ is $0$ on the boundary of $\mathcal X$. Take $\gamma$ from the boundary of $\mathcal X$. We may assume that $|\gamma_2|=1$. Then the elementary calculations give
\begin{equation}
h(\gamma)=|a+b\gamma_1\bar\gamma_2+\gamma_1|-a(1-|\gamma_1|^2)=0.
\end{equation}

The proof of the second inequality goes as follows. Due to the symmetry it is sufficient to show that 
\begin{equation}
|a\gamma_2+b\gamma_1+\gamma_1\gamma_2|+b(1-|\gamma_2|^2)>a(1-|\gamma_1|^2).
\end{equation}
Note that the left side of the above inequality is
\begin{multline}
\left|a\gamma_2-\frac{b}{a}(b\gamma_2+1)-\frac{b\gamma_2+1}{a}\gamma_2\right|+b(1-|\gamma_2|^2)=\\
\left|\frac{b}{a}\gamma_2^2+\frac{b}{a}+\frac{1+b^2-a^2}{a}\gamma_2\right|+b(1-|\gamma_2|^2)\geq\\
\frac{b}{a}|\gamma_2-1|^2-\frac{b^2+2b+1-a^2}{a}|\gamma_2|+b(1-|\gamma_2|^2).
\end{multline}
Therefore, it is sufficient to show that for all $\gamma_2\in\D$ we have
\begin{equation}
\frac{b|1-\gamma_2|^2}{a}-\frac{(b+1)^2-a^2}{a}|\gamma_2|+b(1-|\gamma_2|^2)>\frac{a^2-|b\gamma_2+1|^2}{a}.
\end{equation}
Simplifying above and then dividing by $1+b-a$ the last is equivalent to
\begin{equation}
b|\gamma_2|^2+(a+1)-(b+1+a)|\gamma_2|>0,
\end{equation}
which is equivalent to the inequality
\begin{equation}
(1-|\gamma_2|)(a+1-b|\gamma_2|)>0.
\end{equation}
And the last inequality holds trivially.
\qed
\end{proof}

\begin{remark}\label{remark:blaschke} A small modification of the proof of Lemma~\ref{lem:Viscar} gives a much wider variety of complex geodesics having the components being the Blaschke product of degree one or two. The idea is the following.
For $|\omega|=1$, $\gamma\in\D$ we put
\begin{equation}\label{eq:phi}
\varphi_{\omega, \gamma}(\lambda)=\left(\lambda\frac{\gamma- \omega\lambda}{1-\bar\gamma \omega\lambda}=: \lambda\psi (\lambda), \frac{a\lambda+ b\lambda\psi(\lambda )- \lambda^2\psi(\lambda)}{b\lambda+ a \lambda\psi(\lambda)-1},\lambda\right).
\end{equation} 
Our aim will be to find $\omega$ so that the above mapping lies entirely in $M$ and the second component is the Blaschke product of degree two.

Recall that the Schur algorithm gives that the square polynomial $A\lambda^2+B\lambda+C$ has both roots outside $\overline{\mathbb D}$ if and only if 
$|C|>|A|$ and $|C|^2-|A|^2>|B\bar C-A\bar B|$. Applying this property to our function (coming  from the denominator of the second component) we get the following inequality for $\omega$ that has to be satisfied
\begin{equation}\label{eq:im}
b^2-|a\gamma+1|^2 > |-ab (1-|\gamma|^2) + \omega (a \bar \gamma^2 + (a^2 - b^2 +1) \bar \gamma + a)|
\end{equation}
The above inequality is equivalent to the following one
\begin{equation}
b(1-|\gamma_2|^2)>|a(1-|\gamma_1|^2)-\overline{\omega}(a\gamma_2+b\gamma_1+\gamma_1\gamma_2)|,
\end{equation}
where $\gamma_1:=\gamma$, $\gamma_2:=-\frac{a\gamma_1+1}{b}$. And now we make use of the calculations conducted in the proof of the previous lemma -- to see that the above inequality holds for $|\omega|=1$ lying in a non-empty open arc.

The above reasoning gives the following property for the variety $M$. If only $X\in T_0M$ is such that $|X_1|>|X_j|$, $j=2,3$ we get following
functions 
\begin{equation}
\D\owns\lambda\mapsto(\lambda,\lambda m_{\eta}(\omega\lambda),f_3(\lambda))\in M,
\end{equation}
where $\eta$ is suitably chosen and $|\omega|=1$ may be chosen from some non-empty arc.
\end{remark}

The above result gives an example of a two-dimensional domain for which the infinitesimal version of the Lempert theorem holds. Namely, let
\begin{equation}
D_{a,b}:=\{z\in\D^2:|F_{a,b}(z_1,z_2)|<1\},
\end{equation}
where $F_{a,b}(z_1,z_2):=\frac{az_1+bz_2-z_1z_2}{az_2+bz_1-1}$.

\begin{corollary}\label{corollary:indicatrix} The following equality holds
\begin{equation}
\kappa_{D_{a,b}}\equiv\gamma_{D_{a,b}}.
\end{equation}
In the special case of the origin we have the following formula
\begin{equation}
\kappa_{D_{a,b}}(0;X)=\max\{|X_1|,|X_2|,|aX_1+bX_2|\},\; X\in\mathbb C^2,\; z\in M.
\end{equation}
\end{corollary}

\bigskip

Having proven Lemma~\ref{lem:Viscar} we shall show that $M$ is the Carath\'eodory set. As we already announced the proof will be a consequence of Lemma~\ref{lem:Viscar} and a topological argument.

\begin{proof}[Proof of Theorem~\ref{thm:Viscar}]
We show that
\begin{equation}
c_{M}(0,z)=l_{M}(0,z)=\max\{\rho(0,z_j):j=1,2,3\}.
\end{equation}
Below we use the notation as in Lemma~\ref{lem:Viscar}.
To show the above it is sufficient to show that for any $x\in\D_*$ and any $(\lambda_1,\lambda_2)\in M_x$, where
\begin{equation}
M_x:=\{(\lambda_1,\lambda_2)\in\mathbb D^2: (\lambda_1x,\lambda_2x,x)\in M\}
\end{equation}
there is a $\gamma\in\mathbb \mathcal X$ such that $\Phi_{\gamma}(x)=(\lambda_1x,\lambda_2x,x)$. Denote $\Psi_{\gamma}(x):=\frac{(\Phi_{\gamma}(x))_{1,2}}{x}$. To finish the proof it is sufficient to show that the function
\begin{equation}
\psi_x:\mathcal X\owns \gamma\mapsto\Psi_{\gamma}(x)\in M_x
\end{equation}
is onto. The properties of the functions $(\omega,\eta)$ imply that $\psi_x$ is not only continuous but it also extends continuously onto the closure of $\mathcal X$. Moreover, if $|\gamma_j|=1$ is such that $\gamma\in\overline{\mathcal X}$ then the extension satisfies the equality $(\psi_x(\gamma))_j=\gamma_j$. Both $\mathcal X$ and $M_x$ are simply connected two dimensional surfaces bounded by union of two arcs and such that $\psi_x$ is a homeomorphism between the boundaries. Consequently, $\psi_x$ is onto.

\end{proof}

As a direct consequence of the theorem we get the equality of invariant functions in a class of two-dimensional domains.

\begin{theorem}\label{theorem:lempert-theorem-example} The domain $D_{a,b}$ is a Lempert domain.
\end{theorem}

\begin{remark}\label{remark:not-linearly-convex}
Note that if the triple $\{1,a,b\}$ satisfies the triangle inequality then the domain $D_{a,b}$ is never linearly convex. Actually, to see this take an arbitrary point $w\in\partial D_{a,b}\cap\D^2$. Then $|F_{a,b}(w)|=1$. Assuming the linear convexity we find a vector $v\in\C^2\setminus\{0\}$ such that
$|F_{a,b}(w+\lambda v)|\geq 1$ for $\lambda$ close to zero. The minimum principle for holomorphic functions implies that $F_{a,b}(w+\lambda v)$ equals some unimodular constant for $\lambda$ close to $0$, which in view of the explicit formula for 
$F_{a,b}$ easily implies that either $v_1$ or $v_2$ is zero. Assume that $v_1=0$. Then the function $\lambda\to F_{a,b}(w+\lambda v)$ is a homography, the elementary calculations give that it is constant precisely when 
\begin{equation}
bw_1^2-(b^2+1-a^2)w_1+b=0.
\end{equation}
But the above equation is satisfied only for $|w_1|=1$ (use the triangle inequality for the triple $\{a,b,1\}$ to see that the roots of the above square equation are not reals), which gives the contradiction. 
\end{remark}

\begin{remark} As we saw the Lempert Theorem holds for the bounded hyperconvex domain $D_{a,b}$ which is not linearly convex. It would be desirable to decide whether the domain $D_{a,b}$ is not biholomorphic to a convex domain. Note that the domains $D_{a,b}$
are the next candidates for examples of that kind. Perhaps from the point of view of the Lempert Theory the domains $D_{a,b}$ have much less nice properties than the existing examples of that type (the symmetrized bidisc and tetrablock). In any case it seems reasonable that an effort should be undertaken to understand better the function geometric properties of that class of domains. 
\end{remark}

\begin{remark} Let us underline once more that a study of a class of two-dimensional submanifolds of the tridisc that appeared in the study of the extension property in \cite{Kos-McC 2019} led not only to introducing a new property that seems to be better suited in the study of the extension property (Carath\'eodory sets) but also provides examples of domains interesting for the geometric function theory.
\end{remark}

\section{Universal sets for the (infinitesimal) Carath\'eodory extremal problem}\label{section:universal-set} In \cite{Agl-Lyk-You 2018} the authors introduced the notion of universal sets for the Cara\-th\'e\-o\-do\-ry extremal problem concentrating mainly on the problem in the symmetrized bidisc. In the present section we concentrate on that topic presenting results on the domains admitting finite universal sets.

Let $D$ be a domain in $\mathbb C^n$. We say that $\mathcal C\subset\mathcal O(D.\mathbb D)$ is {\it a universal set for the Carath\'eodory} (respectively, {\it for the infinitesimal Carath\'eodory}) {\it extremal problem} if for any $w,z\in D$, $w\neq z$ (respectively, $w\in D$, $X\in\mathbb C^n$, $X\neq 0$) there is an $F\in\mathcal C$ such that $c_D(w,z)=\rho(F(w),F(z))$ (respectively, $\gamma_D(w;X)=\gamma_{\mathbb D}(F(w);F^{\prime}(w)(X))$).

\begin{remark} Assume that $D$ is $c$-hyperbolic and $\gamma$-hyperbolic (i.e. $c_D(w,z)>0$ for all $w,z\in D$, $w\neq z$ and $\gamma_D(w;X)>0$, $w\in D$, $X\in\C^n\setminus\{0\}$). If $\mathcal C$ is a universal set for the Carath\'eodory extremal problem then $\mathcal C$ is also a universal set for the infinitesimal Carath\'eodory extremal problem.
\end{remark}

As to the problem of the existence of finite universal sets we present a result on one dimensional domains and show that the domains $D_{a,b}$ introduced in Section~\ref{section:sets-M} are examples showing that the situation in dimension one and two differs completely. We also simplify and extend results on characterization of domains with finite universal sets presented in \cite{Agl-Lyk-You 2018}.

Recall that the domain $D$ is {\it $c$-finitely compact } if the Carath\'eodory balls $\{z\in D:c_D(w,z)<r\}$ are relatively compact in $D$ for all $w\in D$, $r>0$.

\subsection{Planar domains with minimal universal sets}
Recall that for planar domains $c$-hyperbolicity is equivalent to $\gamma$-hyperbolicity and this is equivalent to the existence of a non-constant bounded holomorphic function (see e. g. \cite{Jar-Pfl 2013}).

The infinitesimal version of the proposition below is the content of Theorem 1 in \cite{Fis 1969}. The non-infinitesimal case can be obtained along the same lines. 

\begin{proposition}\label{prop:extr-planar}{\rm (cf. Theorem 1 in \cite{Fis 1969}).}
Let $D$ be a planar domain. Then Carath\'eodory extremals and infinitesimal Carath\'eodory extremals are uniquely determined which means that up to automorphisms of the unit disc
for any $w\neq z$ (respectively, $w\in D$) there is only one $F\in\mathcal O(D,\mathbb D)$ such that
\begin{align}
c_D(w,z)&=\rho(F(w),F(z)),\\
\text{ respectively, }\gamma_{D}(w;1)&=\gamma_{\mathbb D}(F(w);F^{\prime}(w)).
\end{align}
\end{proposition}
%\begin{proof}[see the proof of ... in ...] Without loss of generality we may assume that $D$ is $c$- and $\gamma$-hyperbolic.  
%\end{proof}

 We show that in the case of the planar domains the existence of finite universal sets makes the domain (under some evident assumptions) be the unit disc which is contained in the following.

\begin{theorem}\label{thm:extr-planar} Let $D$ be a domain in $\mathbb C$ that has a finite universal set for the infinitesimal Carath\'eodory problem. Then $D$ has a universal set consisting of one element. In particular, if $D$ is additionally $c$-finitely compact then it is biholomorphic to the unit disc.  
\end{theorem}
\begin{proof} Without loss of generality we may assume that $D$ is $\gamma$-hyperbolic. Let $\mathcal C=\{\Phi_1,\ldots,\Phi_N\}$ be a minimal finite universal set for the infinitesimal Carath\'eodory extremal problem. It is sufficient to show that $N=1$. Suppose that $N\geq 2$. Denote $V:=\Phi(D)$. The uniqueness of $\gamma$-extremals and the minimality of $\mathcal C$ imply that for any $z\in D$ there is a $j$ such that
\begin{equation}
\gamma_{\mathbb D}(\Phi_j(z);\Phi_j^{\prime}(z))>\max\{\gamma_{\mathbb D}(\Phi_k(z);\Phi_k^{\prime}(z)),\; k\neq j\}.
\end{equation}
On the other hand the minimality of the set $\mathcal C$ implies that for any $1\leq k\leq N$ there is a $z\in D$ such that 
\begin{equation}
\gamma_{\mathbb D}(\Phi_k(z);\Phi_k^{\prime}(z))>\max\{\gamma_{\mathbb D}(\Phi_j(z);\Phi_j^{\prime}(z)):j\neq k\}.
\end{equation}
Standard connectivity argument shows, however, that both statements cannot hold simultanuously.
\end{proof}

\subsection{Finite universal sets induce embeddings into polydiscs} 
Below we generalize Theorem 2.3 from \cite{Agl-Lyk-You 2018} with a simpler proof.

\begin{theorem}\label{the:univer-fin} Let $D\subset\mathbb C^n$ be $c$-hyperbolic and $\gamma$-hyperbolic. Assume additionally that $\mathcal C=\{\Phi_1,\ldots,\Phi_N\}$ is a universal set 
for the Carath\'eodory extremal problem. Then the mapping
\begin{equation}
\Phi:=(\Phi_1,\ldots,\Phi_N):D\to\mathbb D^N
\end{equation}
is a holomorphic embedding. In particular, $N\geq n$ and $\Phi(D)$ is a connected complex submanifold of dimension $n$. 

Moreover, $\Phi(D)$ is a Carath\'eodory set. In particular we get
\begin{align}
c_{\Phi(D)}(w,z)&=\max\{\rho(\Phi_j(w),\Phi_j(z)):j=1,\ldots,N\},\\
\gamma_{\Phi(D)}(\Phi(w);\Phi^{\prime}(w)(X))&=\max\{\gamma_{\mathbb D}(\Phi_j(w);\Phi_j^{\prime}(w)(X):j=1,\ldots,N\}
\end{align} 
$w,z\in D$, $X\in\mathbb C^n$.

If $D$ is additionally $c$-finitely compact then $\Phi$ is proper so in the case $n=N$ we get that $\Phi(D)=\mathbb D^n$.

\end{theorem}
\begin{proof}
The definition of the universal Carath\'eodory set gives 
\begin{multline}
c_D(w,z)=\max\{\rho(\Phi_j(w),\Phi_j(z)):j=1,\ldots,N\}=\\
c_{\mathbb D^N}(\Phi(w),\Phi(z)), \; w,z\in D.
\end{multline}
Let $w,z\in D$ be such that $\Phi(w)=\Phi(z)$. Then $c_D(w,z)=0$ and the $c$-hyperbolicity implies that $w=z$. Therefore, $\Phi$ is injective.
Similarly, because of the fact that $\{\Phi_1,\ldots,\Phi_n\}$ is a universal set for the infinitesimal Carath\'eodory extremal problem we get
\begin{multline}
\gamma_D(w;X)=\max\{\gamma_{\mathbb D}(\Phi_j(w);\Phi_j^{\prime}(w)(X)):j=1,\ldots,N\}=\\
\gamma_{\D^N}(\Phi(w);\Phi^{\prime}(w)(X))\},\, w\in D, X\in\mathbb C^n.
\end{multline}
Since $D$ is $\gamma$-hyperbolic we get that the rank of $\Phi^{\prime}(w)$ is $n$. 

We prove that $V:=\Phi(D)$ is a Carath\'eodory set. This can be seen as follows.
\begin{multline}
c_{\mathbb D^N}(\Phi(w),\Phi(z))=\max\{\rho(\Phi_j(w),\Phi_j(z)):j=1,\ldots,N\}=\\c_D(w,z)=
c_{\Phi(D)}(\Phi(w),\Phi(z))\geq c_{\mathbb D^N}(\Phi(w),\Phi(z)).
\end{multline}

Assume that $D$ is additionally $c$-finitely compact. We show below that $\Phi$ is proper.

Fix $w\in D$ and let $(z^k)_k$ have no accumulation point in $D$. Then the equality
\begin{equation}
c_{\mathbb D^N}(\Phi(w),\Phi(z^k))=c_D(w,z^k)\to\infty
\end{equation}
implies that $(\Phi(z^k))_k$ has no accumulation point in $\mathbb D^N$ which gives the desired properness of $\Phi$.
\end{proof}

\subsection{Not only polydisc admits finite universal sets for the Cara\-th\'e\-o\-do\-ry extremal problem in higher dimension} It is shown in \cite{Agl-Lyk-You 2018} that the projections are contained in any universal Carath\'eodory set of the bidisc. It turns out that under evident assumptions domains $D_{a,b}$ have similar property - the three functions defining $D_{a,b}$ must lie in any universal Carath\'eodory set. Moreover, the domains $D_{a,b}$ are examples of (very nice, for instance $c$-finitely compact) domains which admit the finite universal Carath\'eodory sets and still not being the bidisc. This shows that the situation in the case $n=1$ differs from the case of $n$ bigger. The fact that $D_{a,b}$ is not biholomorphic to the bidisc follows for instance from the fact that the indicatrix of $D_{a,b}$ at $0$ is the domain (see Corollary~\ref{corollary:indicatrix})
\begin{multline}
\{X\in\C^2:\kappa_{D_{a,b}}(0;X)<1\}=\\
\{X\in\C^2:\max\{|X_1|,|X_2|,|aX_1+bX_2|\}<1\},
\end{multline}
which is not linearly isomorphic to the bidisc under the assumption that the triple $\{a,b,1\}$ satisfies the triangle inequality.

\begin{example}\label{example:finite-universal-set} Recall that Agler, Lykova and Young remarked that in the bidisc any universal set for the Carath\'eodory problem must contain (up to an automorphism) both projections. Similar property holds for the domains $D_{a,b}$. More precisely, if the triple $\{a,b,1\}$ satisfies the triangle inequality then
the universal set for the Carath\'eodory extremal problem for the domain $D_{a,b}=\{z\in \D^2:\ |az_1+ bz_2 - z_1 z_2|<|bz_1 + az_2 -1|\}$ contains (up to automorphisms)  three functions: 
\begin{equation} z\mapsto z_1,\ z\mapsto z_2 \text{ and } z\mapsto \frac{az_1 + bz_2 - z_1 z_2}{bz_1 + az_2 -1}. 
\end{equation}
To show above note that it is sufficient to see that any of the functions $z_j$ (up to an automorphism), $j=1,2,3$, must belong to a universal set for the Carath\'eodory problem in $M_{\alpha}$ being not a retract. Take such a variety. Consider the functions as in Remark~\ref{remark:blaschke} 
\begin{equation}
\D\owns\lambda\mapsto(\lambda,\lambda m_{\eta}(\omega\lambda),f_3(\lambda))\in M,
\end{equation}
where $\eta$ is fixed and $\omega$ is from some non-empty arc. Then one of left inverses (call it $F$) of the function for the fixed $\omega$ is a left inverse for all $\omega$ from the given arc. Consequently, the function $F$ depends only on the first variable and this equals $z_1$.

%To show it denote by $\Phi$ one of the functions listed above. We need to show that that it must belong to the universal Carath\'eodory set for $D$. %Permuting the variables or substituting $z_2'=\frac{az_1 + bz_2 - z_1 z_2}{bz_1 + az_2 -1}$ we can assume that $\Phi(z) = z_1$. It follows from Remark~%\ref{remark:blaschke} that for any $\gamma_1,\gamma_2\in\D$ with $a\gamma_1+b\gamma_2+1=0$ and $\omega_0\in \TT$ such that 
%$b(1- |\gamma_2|^2)> |a(1 - |\gamma_1|^2)  - \bar\omega (a \gamma^2 + b \gamma_1 + \gamma_1\gamma_2)|$ the disc $\lambda \mapsto 
%(\lambda, \lambda m_\gamma (\omega \lambda))$ is a complex geodesic in $\D$. Fix $\omega_0$ and $\omega_0$ satisfying the above-mentioned %inequalities.

%Let $F:D\to \D$ be a function that is in the universal Carath\'eodory set for $D$ such that, up to a unimodular constant, $F(\lambda, \lambda %m_{\gamma_0} (\omega_0 \lambda)) = \lambda,$ $\lambda\in \D$. Then $F(\lambda, \lambda m_{\gamma_0} (\omega \lambda)) = \lambda,$ for $%\lambda\in \D$ and $\omega\in \T$ that is sufficiently close to $\omega_0$. From this we get that $F$ depends only on the firs variable and $F(z)=z_1$, $z\in %D$.
\end{example}

\begin{remark} It follows from \cite{Kos-McC 2019} that any subdomain of $\C^2$ that has a universal set composing of three elements is biholomorphic to a submanifold of $\D^3$ that is a graph of a holomorphic function for each choice of the coordinates. In particular, it is biholomorphic to a domain of the form $D=\{(z_1, z_2)\in \D^2: h(z_1,z_2)<1\}$, where $h$ is a holomorphic function such that 
$z_1\mapsto h(z_1, z_2)$ (respectively $z_2\mapsto h(z_1, z_2)$) is injective  for every $z_2\in \D$ (resp. $z_1\in \D$). Some other properties of $h$ were obtained in \cite{Kos-McC 2019}.  Recall that the results in \cite{Kos-McC 2019} are stated with the additional assumption of polynomial convexity. However, here they can be dropped out. This forces two natural questions. The first one is if any domain with a $3$-element (or finite) universal set comes from the variety $M_{\alpha}$. 

The second one is more particular; namely, whether any domain having finite universal set is a Lempert domain.
\end{remark}

\section{Universal sets for the Carath\'eodory extremal problem for the unit ball $\mathbb B_2$}\label{section:ball}

%\subsection{Proper holomorphic embeddings into polydiscs}
%As a completion of Theorem~\ref{the:univer-fin} note that if there is a holomorphic embedding $\Phi:D\to\mathbb D^N$ and $\Phi(D)$ is the Carath\'eodory %set then the domain $D$ admits a finite universal set for the Carath\'eodory extremal problem (with $\mathcal C=\{\Phi_1,\ldots,\Phi_N\}$). It seems %reasonable to formulate the following problem.

%{\bf Problem.} Is it true that under some mild assumptions any Carath\'eodory set in the polydisc is a proper holomorphic embedding of a lower dimensional %domain with a finite universal set for the Carath\'eodory extremal problem?
 
%\bigskip
In the last section we make some remarks on the universal sets in the unit ball. The results presented may be seen as the starting point for the further study of a possible structure of (in some at the moment not well determined sense) small universal sets for the Carath\'eodory extremal problem.
 
We start with an easy observation that the unit ball $\mathbb B_n$, $n>1$, does not have a finite universal set for the Carath\'eodory extremal problem. 

\begin{proposition}
The unit ball $\mathbb B_n$, $n>1$, does not have a finite universal set for the Carath\'eodory problem.
\end{proposition}
\begin{proof} It is a direct consequence of Lemma 5 from \cite{Kos-Zwo 2018} which states the following.

Any two different complex geodesics of the ball passing through $0$ have different (up to automorphisms of the unit disc) left inverses.
\end{proof}

\begin{remark} The result on the existence of many complex geodesics passing through $0$ in the symmetrized bidisc and tetrablock that admit only one (up to an automorphism of the unit disc) left inverse can be found in \cite{Kos-Zwo 2016}. The fact that the symmetrized bidisc does not have a finite universal set follows from Theorem 3.1 in \cite{Agl-Lyk-You 2018}. Consequently, the same holds for the tetrablock.
\end{remark}

%Note that the existence of proper holomorphic mappings between $\mathbb B_n$ and polydiscs $\mathbb D^N$ with sufficiently big $N$ is  shown in 
%\cite{Low 1986}. This embedding gives us the non-trivial example of the set that is not the Carath\'eodory set in the polydisc $\D^N$. It would be nice if we %could also find proper holomorphic embeddings of the symmetrized bidisc and/or tetrablock in higher dimensional polydiscs.

%\begin{remark} One should be able to prove the above results for strongly pseudoconvex domains.
%\end{remark}

%Similarly as above we get the non-existence of proper holomorphic embeddings of the symmetrized bdisc in the polydisc.

%More precisely, the result below is a simple consequence of the existence of many complex geodesics passing through $0$ in the symmetrized bidisc and %tetrablock and admitting only one (up to an automorphism of the unit disc) that can be found in \cite{Kos-Zwo 2016}.
%\begin{proposition}
%The symmetrized bidisc and the tetrablock do not admit a finite universal set for the Carath\'eodory problem. Consequently, there are not proper %holomorphic embeddings of these domains into polydiscs (of higher dimension).
%\end{proposition}

The most standard and natural procedure producing a class of the universal set for the Carath\'eodory extremal problem in the unit ball $\mathcal B_n$ is following. As the unit ball is an example of a Lempert domain both notions: extremals and infinitesimal Carath\'eodory extremals coincide. Moreover, the extremals are precisely the ones being the left inverses to complex geodesic, which in turn are parametrizations of portions of complex lines lying in the unit ball. To produce extremal to one of complex geodesic (represented by $l\cap\mathbb B_n$) we may proceed as follows. Let $a\in l\cap\mathbb B_n$ be the point of the minimal norm. Let $\Phi_a$ be the automorphism of the unit ball such that $\Phi_a(a)=0$, $\Phi_a(0)=a$ -- recall that 
\begin{equation}
\Phi_a(z):=.\frac{\sqrt{1-||a||^2}(\langle z,a\rangle a-||a||^2z)-\langle z,a\rangle a+||a||^2a}{||a||^2(1-\langle z,a\rangle)},\; z\in\mathbb B_n
\end{equation}
for all $a\in\mathbb B_n\setminus\{0\}$ and $\Phi_0:=\operatorname{id}_{\mathbb B_n}$.

Now we may apply the unitary mapping $U$ such that $U(\Phi_a(l\cap\mathbb B_n))=\mathbb D\times\{0\}^{n-1}$. And now the mapping $\Psi_l:=U_1\circ\Phi_a$ is one of possible Carath\'eodory extremals for the points from the geodesic $l$. 

In the simplest case of $n=2$ we may apply the above method and we see that the universal set for the Carath\'eodory extremal problem may be chosen in the following way.
\begin{multline}
\mathcal C:=\left\{\frac{\sqrt{1-||a||^2}(\bar a_1z_2-\bar a_2z_1)}{||a||(1-(\bar a_1z_1+\bar a_2z_2))}:a=(a_1,a_2)\in\mathbb B_2\right\}\cup\\
\left\{a_1z_1+a_2z_2:a_1\geq 0,|a_2|^2=1-a_1^2\right\}.
\end{multline}
Let us make one more remark on the properties of the above construction. It follows directly from the construction of $\Psi_l$ that it is a rational mapping that is holomorphic on a neghborhood of $\overline{\mathbb B_n}$ and $\Psi_l^{-1}(\partial \mathbb D)\cap\overline{\mathbb B_n}=l\cap\partial\mathbb B_n$. In other words the universal set for the Carath\'eodory extremal problem $\mathcal C$ just defined is parametrized by complex lines $l$ intersecting $\mathbb B_n$ and it is minimal in the sense that no proper subset of $\mathcal C$ is a universal set. Moreover, any extremal mapping is a left inverse to the unique geodesic. 

It is natural to see whether one could define a class of Carath\'eodory extremals that could be extremals for a wider variety of geodesics, which would yield then a universal set being 'smaller' than the one produced above. Following the construction of extremals from a recent paper \cite{Kos-Zwo 2018}   we shall below get the desired class of functions. We restrict ourselves for the dimension $n=2$.

Let 
\begin{equation}
F(z_1,z_2):=\frac{2z_1(1-z_1)-z_2^2}{2(1-z_1)-z_2^2},\; z\in\mathbb B_2.
\end{equation}
Recall that $F\in\mathcal O(\mathbb B_2,\mathbb D)$ and $F(z_1,0)=z_1$, $z_1\in\mathbb D$ (see \cite{Kos-Zwo 2018}).
\begin{remark} Note that the mapping $F$ just defined assumes the value of absolute value one on a bigger portion of $\partial\mathbb B_2$ than that of the most natural form of extremal mappings ($\Psi_l$ from the previous Section). Actually, note that elementary calculations give the property that for $z\in\partial\mathbb B_2$ the equality $|F(z)|=1$ if and only if
	\begin{equation}
	\operatorname{Im}(z_2(1-\bar z_1))=0.
	\end{equation}
\end{remark}
The above property (fact that the absolute value equal to one is assumed at two dimensional subset of $\partial\mathbb B_2$ suggests that the function $F$ may be a left inverse to a one-dimensional family of complex geodesics). And this is really the case as the next observation shows.

We claim  that for any $t\in\mathbb R$ the mapping $F$ is the left inverse to the mapping (complex geodesic in $\mathbb B_2$)
\begin{equation}
f_t(\lambda):=\left(\frac{t^2+\lambda}{1+t^2},\frac{t(\lambda-1)}{1+t^2}\right),\; \lambda\in\mathbb D.
\end{equation} 
Recall that ($c_D^*:=\operatorname{arctanh} c_D$)
\begin{equation}
c_{\mathbb B_2}^*(w,z)=\sqrt{1-\frac{(1-||w||^2)(1-||z||^2)}{|1-\langle w,z\rangle|^2}},\; w,z\in\mathbb B_2.
\end{equation}
It is elementary then to check that $f_t(\mathbb D)\subset\mathbb B_2$. It is therefore sufficient to show that
\begin{equation}
c_{\mathbb B_2}(F(f_t(\lambda_1),F(f_t(\lambda_2)))=p(\lambda_1,\lambda_2),\;\lambda_1,\lambda_2\in\mathbb D,
\end{equation}
which follows directly form the formula for the Carath\'eodory distance for the unit ball and the Poincar\'e distance.

\textbf{Acknowledgments.} The authors wanted to thank the anonymous referee for comments and corrections that improved the shape of the paper.


\begin{thebibliography}{10}

%\bibitem{Abate} \textsc{M.~Abate}, \textit{The complex geodesics of non-hermitian symmetric spaces}, Universiti degli Studi di %Bologna, Dipartamento di Matematica, Seminari di geometria, 1991-1993, 1--18; the paper available at %\verb'www.dm.unipi.it/~abate/articoli/artric/files/CompGeodHermSymSp.pdf'.

%\bibitem{Abo 2007} \textsc{A.~A.~Abouhajar}, \textit{Function theory related to $H^{\infty}$ control}, Ph.D. thesis, Newcastle %University, 2007.

\bibitem{AWY} \textsc{A.~A.~Abouhajar, M.~C.~White, N.~J.~Young}, \textit{A Schwarz lemma for a domain related to  mu-synthesis}, 
Journal of Geometric Analysis, 17(4), 2007, 717-750.

%\bibitem{Agl-Lyk-You 2013} \textsc{J.~Agler, Z.~Lykova, N.~J.~Young}, \textit{Extremal holomorphic maps and the symmetrized bidisc}, Proc. London Math. %Soc. (3) 106 (2013) 781-818.

%\bibitem{Agl-Lyk-You2 2013} \textsc{J.~Agler, Z.~Lykova, N.~J.~Young}, \textit{$3$-extremal holomorphic maps and the %symmetrised bidisc}, J. Geom. Analysis, to appear (2013).

\bibitem{Agl-Lyk-You 2018} \textsc{J. Agler, Z. Lykova, N. J. Young}, \textit{Characterizations of some domains via Carath\'eodory extremals}, J. Geometric Analysis, 29(4), 2019, 3039--3054.

\bibitem{Agl-Lyk-You 2019} \textsc{J. Agler, Z. Lykova, N. J. Young}, {\it Geodesics, retracts, and the norm-preserving extension property in the symmetrized bidisc}, Memoirs of the American Mathematical Society 2019, 258, no. 1242, 106pp.

\bibitem{Agl-McC 2003} \textsc{J. Agler and J.E. McCarthy}, {\it Norm preserving extensions of holomorphic
functions from subvarieties of the bidisk}, Ann. of Math., 157(1), 289--312,
2003.

\bibitem{Agl-You 2004} \textsc{J.~Agler, N.~J.~Young}, \textit{The hyperbolic geometry of the symmetrized bidisc},  J. Geom.  Anal.  14  (2004),  no. 3, 375--403.

%\bibitem{Ama-Tho1994} \textsc{E.~Amar, P.~J.~Thomas}, \textit{A notion of extremal discs related to interpolation in the Ball}, %Math. Annalen, 300 (1994) 419-433.

%\bibitem{Agl} {\sc J.~Agler, N.~J.~Young}, {\it The two-by-two spectral Nevanlinna-Pick problem}, Trans. Amer. Math. Soc. 356 %(2004) 573--585.

%\bibitem{A-Y1} J.~Agler \& N.~J.~Young, \textit{A Schwarz lemma for the symmetrized
%bidisc},  Bull. London Math. Soc.  33  (2001), 175--186.

%\bibitem{Agl-You 2006} J.~Agler \& N.~J.~Young \textit{The complex geodesics of the
%symmetrized bidisc}  Internat. J. Math.  17  (2006), 375--391.

%\bibitem{Aro-Fal-Mae 2017} \textsc{R. M. Aron, J. Falc\'o, M. Maestre}, \textit{Separation theorems for group invariant
%polynomials}, Journal of Geometric Analysis, to appear, 2017.

%\bibitem{Bas} \textsc{G.~Bassanelli}, \textit{On holomorphic automorphisms of the unit ball of rectangular, complex matrices.}  Ann. Mat. Pura Appl. (4)  133  (1983), 159--175.


\bibitem{Bha} \textsc{T. Bhattacharyya, H. Sau}, {\it Holomorphic functions on the symmetrized bidisk—realization, interpolation and extension}, J. Funct. Anal. 274 (2018), no. 2, 504--524.
 
%\bibitem{Car 1957} \textsc{H. Cartan}, \textit{Quotient d’un espace analytique par un groupe d’automorphismes}, in \textit{Algebraic geometry
%and topology}, 90--102, Princeton University Press, Princeton, N. J. (1957).

%\bibitem{Car 1960} \textsc{H. Cartan}, \textit{Quotients of complex analytic spaces. Contributions to function theory}, Internat. Colloq.
%Function Theory, Bombay, 1960, 1--15, Tata Institute of Fundamental Research, Bombay.

%\bibitem{Cha-Gor 2015} \textsc{D. Chakrabarti, S. Gorai}, \textit{Function theory and holomorphic maps on symmetric products of planar
%domains}, J. Geom. Anal. 25 (2015), no. 4, 2196--2225.

%\bibitem{Cha-Gro 2017} \textsc{D. Chakrabarti, C. Grow}, \textit{Arch. Math}, (2017), to appear, https://doi.org/10.1007/s00013-017-1091-7.

%\bibitem{Chi 1989} \textsc{E. M. Chirka}, \textit{Complex Analytic Sets}, Kluwer Academic Publishers, 1989.


%\bibitem{Cos 2004} \textsc{C.~Costara},  \textit{Le proble\'me de Nevanlinna-Pick spectral}, Ph.D. thesis, Universit\'e Laval, 2004.

\bibitem{Cos 2004} \textsc{C.~Costara}, \textit{The symmetrized bidisc and Lempert's theorem}, Bull. London Math. Soc. 36 (2004), no. 5, 656--662.

%\bibitem{Edi1} A.~Edigarian, Lempert

%\bibitem{Edi 2004} \textsc{A.~Edigarian}, \textit{ A note on C. Costara's paper: ``The symmetrized bidisc and Lempert's theorem'' [Bull. London Math. Soc. 36 (2004), no. 5, 656--662]}, Ann. Polon. Math. 83 (2004), no. 2, 189--191.

%\bibitem{Edi 1995} \textsc{A.~Edigarian}, \textit{On extremal mappings in complex ellipsoids},
%Ann. Polon. Math. 62 (1995), 83--96.

\bibitem{Edi-Kos-Zwo 2013} \textsc{A.~Edigarian, \L. Kosi\'nski, W.~Zwonek}, \textit{The Lempert Theorem and the Tetrablock}, Journal of Geom. Anal., 23 (2013), no. 4, 1818-–1831.

%\bibitem{Edi-Zwo 2005} \textsc{A.~Edigarian, W.~Zwonek}, \textit{Geometry of the
%symmetrized polydisc}  Arch. Math. (Basel)  84  (2005),  no. 4,
%364--374.

\bibitem{Fis 1969} \textsc{S.D. Fisher}, \textit{On Schwarz's lemma and inner functions}, Trans. Amer. Math. Soc, 138 (1969),
229--240.

%\bibitem{For} J.E.~Fornaess,

\bibitem{Guo-Hua-Wan 2008} \textsc{Kunyu Guo, Hansong Huang, and Kai Wang}, \textit{Retracts in polydisk and
analytic varieties with the $H^{\infty}$-extension property}, J. Geom. Anal.,
18(1):148--171, 2008.

%\bibitem{Har} {\sc L. A. Harris}, {\it Bounded symmetric homogeneous domains in infinite dimensional spaces}, Proceedings on %Infinite Dimensional Holomorphy, 13–-40. Lecture Notes in Math., Vol. 364, Springer, Berlin, 1974. 

%\bibitem{Hua} \textsc{L.K.~Hua}, \textit{Harmonic Analysis of Functions of Several Complex Variables in the Classical Domains}, %AMS. (1963), Providence.

%\bibitem{Gun-Ros 2009}  \textsc{R. C. Gunning, H. Rossi}, \textit{Analytic functions of several complex variables}, reprint of the
%1965 original. AMS Chelsea Publishing, Providence, RI, 2009.

\bibitem{Hea-Suf 1981} \textsc{L. F. Heath and T. J. Suffridge}, \textit{Holomorphic retracts in complex $n$-space} Illinois J. Math., 25(1), 125--135, 1981.

%\bibitem{Jac 2006} \textsc{D.~Jacquet}, \textit{$\mathbb C$-convex domains with $C^2$ boundary}, Complex Var. Elliptic Equ. 51 (2006), no. 4, 303--312.

\bibitem{Jar-Pfl 2013} \textsc{M.~Jarnicki, P. Pflug}, \textit{Invariant Distances and Metrics in Complex Analysis - 2nd exteded edition}, De Gruyter Expositions in Mathematics 9, 2013.

%\bibitem{Jar-Pfl 2000} \textsc{M.~Jarnicki, P. Pflug}, \textit{Extension of Holomorphic Functions}, de Gruyter Expositions in %Mathematics 34, Walter de Gruyter, 2000.

%\bibitem{Kne 2007} \textsc{G. Knese}, \textit{Function theory on the Neil parabola}, Michigan Math. J., Vol. 55, Issue 1 (2007), 139--154.

%\bibitem{Kos 2009} \textsc{\L.~Kosi\'nski}, \textit{Geometry of quasi-circular domains and applications to tetrablock}, Proc. Amer. %Math. Soc. 139 (2011), 559--569.

\bibitem{Kos 2015} \textsc{\L. Kosi\'nski}, \textit{Three-point Nevanlinna-Pick problem in the polydisc}, Proc. London Math. Soc., 111(4) (2015), 887-910.

%\bibitem{Kos-Tho-Zwo 2015} \textsc{\L. Kosi\'nski, P. J. Thomas, W. Zwonek}, \textit{Coman conjecture for the bidisc}, http://arxiv.org/abs/1411.4322.

%\bibitem{Kos-Zwo 2015} \textsc{\L. Kosi\'nski, W. Zwonek}, \textit{Extremal holomorphic maps in special classes of domains}, http://arxiv.org/abs/%1401.1657, Ann. Sc. Norm. Super. Pisa, to appear.


\bibitem{Kos-McC 2019} \textsc{\L. Kosi\'nski, J.E. McCarthy}, \textit{Extensions of Bounded Holomorphic Functions on the Tridisc}, Rev. Mat. Iberoam (2019).

\bibitem{Kos-McC 2018} \textsc{\L. Kosi\'nski, J.E. McCarthy}, \textit{Norm preserving extensions of bounded holomorphic functions}, Trans. Amer. Math. Society 371 (10) (2019), 7243--7257

\bibitem{Kos-Zwo 2016} \textsc{\L. Kosi\'nski, W. Zwonek},  \textit{Nevanlinna--Pick Problem and Uniqueness of Left Inverses in Convex Domains, Symmetrized Bidisc and Tetrablock}, J. Geometric Analysis, 26 (2016), 1863--1890.

\bibitem{Kos-Zwo 2018} \textsc{\L. Kosi\'nski, W. Zwonek}, \textit{Nevanlinna-Pick interpolation problem in the ball}, Trans. Amer. Math. Society,  370 (2018), 3931--3947.


%\bibitem{Kra} \textsc{S.~Krantz}, \textit{Function Theory of Several Complex Variables: Second Edition},AMS Chelsea Publishing (1992).

\bibitem{Lem 1981} \textsc{L.~Lempert}, \textit{La m\'etrique de Kobayashi et la repr\'esentation des domaines sur la boule}, Bull. Soc. Math. France 109 (1981), no. 4, 427--474.

\bibitem{Lem 1982} \textsc{L.~Lempert}, \textit{Holomorphic retracts and intrinsic metrics in convex domains}, Anal. Math. 8 (1982), no. 4, 257--261.

%\bibitem{Loj 1991} \textsc{S.~\L ojasiewicz}, \textit{Introduction to Complex Analytic Geometry}, Birkh\"auser Verlag, 1991.

%\bibitem{Low 1986} \textsc{E. Low}, \textit{The unit ball is a closed complex submanifold of a polydisk}, Inventiones mathematicae, 83, 405--410 (1986).

\bibitem{Mac 2019} \textsc{K. Maciaszek}, \textit{On Polynomial Extension Property in $N$-Disc}, Complex Anal. Oper. Th.,  13(6) (2019), 2771--2780.

%\bibitem{Nik-Pfl-Zwo 2008} \textsc{N.~Nikolov, P.~Pflug \& W.~Zwonek}, \textit{An example of a bounded $\CC$-convex domain which is not biholomorphic t%o a convex domain}, Math. Scand. 102 (2008), no. 1, 149--155.

%\bibitem{NPTZ 2008} N.~Nikolov, P.~Pflug, P.~J.~Thomas, W.~Zwonek,
%\textit{Estimates of the Carath\'eodory metric of the symmetrized
%polydisc}, J. Math. Anal. Appl. 341 (2008), 140--148.

%\bibitem{Pfl-Zwo 1998} P.~Pflug, W.~Zwonek, \textit{Effective formulas for
%invariant functions---case of elementary Reinhardt domains}  Ann.
%Polon. Math.  69  (1998),  no. 2, 175--196.

%\bibitem{Pfl-Zwo 2005} P.~Pflug \& W.~Zwonek,
%\textit{Description of all complex geodesics in the symmetrized
%bidisc.}, Bull. London Math. Soc.  37  (2005), 575--584.

\bibitem{Pfl-Zwo 2012} \textsc{P. Pflug, W. Zwonek}, \textit{Exhausting domains of the symmetrized bidisc},
Ark. Mat., 50(2), (2012), 397-402.

%\bibitem{Pick} {\sc G.~Pick}, {\it \"Uber die Beschr\"ankungen analytischer Funktionen, welche durch vorgegebene
%Funcktionswerte bewirkt werden}, Math. Ann. 77 (1916) 7–23.

%\bibitem{Pol} E.~Poletsky

%\bibitem{Rah 2002} \textsc{Q.~I.~Rahman \& G.~Schmeisser}, \textit{Analytic Theory of Polynomials}, London Math. Soc. Monogr. (N.S.) 26, Oxford Univ. Press, Oxford, 2002.

\bibitem{Rud 1969} \textsc{W. Rudin}, \textit{Function Theory in Polydiscs}, Benjamin, New York, 1969. 

%\bibitem{You 2007} \textsc{N.~J.~Young}, \textit{The automorphism group of the tetrablock}, Journal of the London %Mathematical Society  77(3) (2008), 757-770.

%\bibitem{Two} {\sc P.~Tworzewski}, {\it Intersections of analytic sets with linear subspaces}, Annali della Scuola Normale Superiore %di Pisa - Classe di Scienze, Sér. 4, 17 no. 2 (1990), p. 227-271.

%\bibitem{War 2015} \textsc{T. Warszawski}, \textit{(Weak) $m$-extremals and $m$-geodesics}, Complex Var. Elliptic Eq.,
%60(8) (2015), 1077-1105.
 
%\bibitem{Whi 1972} \textsc{H. Whitney}, \textit{Complex Analytic Varieties}, Addison-Wesley Series in Mathematics, Addison-Wesley,
%1972.

%\bibitem{Zna 2001} \textsc{S.~V.~Znamenskii}, \textit{Seven problems on $\mathbb C$-convexity}, Complex analysis in modern mathematics (Russian), %123--131, FAZIS, Moscow, 2001.

%\bibitem{Zwo} W.~Zwonek \textit{Completeness, Reinhardt domains and the
%method of complex geodesics in the theory of invariant functions},
%Dissertationes Math. 388  (2000), 103 pp.

%\bibitem{Zwo 2013} \textsc{W.~Zwonek}, \textit{Geometric properties of the tetrablock}, Arch. Math. 100(2) (2013), 159--165.

\end{thebibliography}
\end{document}